\documentclass{amsart}
\numberwithin{equation}{section}

\newtheorem{thm}{Theorem}[section]
\newtheorem{theorem}[thm]{Theorem}
\newtheorem{lemma}[thm]{Lemma}
\newtheorem{corollary}[thm]{Corollary}

\theoremstyle{definition}

\newtheorem*{theorem*}{Theorem}

\newtheorem{remark}[thm]{Remark}

\newtheorem{example}[thm]{Example}

\newtheorem{defn-thm}[thm]{Definition-Theorem}

\newcommand{\sF}{{\mathcal F}}

\newcommand{\sL}{{\mathcal L}}

\newcommand{\N}{{\mathbb N}}

\newcommand{\End}{{ End}}

\newcommand{\Hom}{{ Hom}}

\renewcommand{\bar}{\overline}

\renewcommand{\phi}{\varphi}

\newcommand{\ee}{\end{eqnarray*}}
\newcommand{\be}{\begin{eqnarray*}}

\newcommand{\beq}{\begin{equation}}
\newcommand{\eeq}{\end{equation}}

\newcommand{\bd}{\begin{enumerate}}
\newcommand{\ed}{\end{enumerate}}

\renewcommand{\hat}{\widehat}

\usepackage{fullpage}

\def\ZZ{{\mathbb{Z}}}
\def\QQ{{\mathbb{Q}}}

\def\CC{{\mathbb{C}}}
\def\caL{{\mathcal{L}}}
\def\cO{{\mathcal{O}}}
\def\cA{{\mathcal{A}}}

\usepackage{amsmath}
\usepackage{amsfonts}
\usepackage{amssymb}
\usepackage{xypic}
\usepackage{longtable}
\usepackage{amsthm}
\usepackage{booktabs}
\usepackage{caption}
\usepackage{color}

\DeclareMathOperator{\degi}{deg}

\DeclareMathOperator{\Pic}{Pic}
\DeclareMathOperator{\Endo}{End}
\DeclareMathOperator{\Ima}{Im}
\DeclareMathOperator{\diag}{diag}

\begin{document}

\title{A note on dual abelian varieties}
\author[1]{Aleksandra Bor\'owka}
\address[1]{Institute of Mathematics, Jagiellonian University in Krak\'ow\\Poland}
\email{aleksandra.borowka@uj.edu.pl}

\author[2]{Pawe\l{} Bor\'owka}
\address[2]{Institute of Mathematics, Jagiellonian University in Krak\'ow\\Poland}
\email{pawel.borowka@uj.edu.pl}

\maketitle

\begin{abstract}
For any non-principal polarisation $D$, we explicitly construct $D$-polarised abelian variety $A$, such that its dual abelian variety is not (abstractly) isomorphic to $A$. For $\dim(A)>3$ the construction includes examples with submaximal Picard number equal to $(\dim(A)-1)^2+1$. As a corollary, we show that a very general non-principally polarised abelian variety is not isomorphic to its dual. Moreover, we show an example of an abelian variety that is isomorphic to its dual, yet it does not admit a principal polarisation.
\end{abstract}

\section{Introduction}
An $n$-dimensional complex torus $A$ is called an abelian variety if it is equipped with an ample line bundle, denoted $\sL$. Its Neron-Severi class $c_1(\caL)$ is called a polarisation and when we fix it, we call the variety polarised. By \cite[Prop. 2.5.3 and 2.5.4]{BL} we can think of a polarisation as a choice of an ample line bundle up to translation, hence we will abuse the notation by saying that $\sL$ is a polarisation. Having $\sL$ we can consider a polarising isogeny, i.e. the map (that depends only on the Neron-Severi class)
$\varphi_{\sL}:A\to\Pic^0{A}$ given by $\varphi_{\sL}(x)=t^*_x(\sL)\otimes \sL^{-1}$, where $t_x$ is translation by $x$.
Since its kernel is a finite abelian group, one can prove that it is isomorphic to $(\ZZ_{d_1}\times\ldots\times\ZZ_{d_n})^2$ where $d_i$ are positive integers satisfying $d_i|d_{i+1},\ i=1,\ldots,n-1$. The sequence $(d_1,\ldots,d_n)$ is called the type of polarisation and one can always find an ample line bundle for which $d_1=1$. 
Using the fact that $\Pic^0(A)=\hat{A}$ is the dual abelian variety, one sees that if $\sL$ is principal (i.e. of type $(1,\ldots,1)$) then we have an isomorphism $A\cong \hat{A}$.

A usual elementary example of an abelian variety is an elliptic curve $E$. One easily checks that for any point $P\in E$ we have that  the class of $\cO_E(P)$ is principal and hence $E\cong \hat{E}$ for any abelian variety of dimension 1.  
Then, a natural problem arises, namely to find abelian varieties that are not principally polarised and are not isomorphic to their duals. We have two kind of easy solutions to the problem. Both of them start with the existence of an abelian variety such that its endomorphism ring $\End(A)$ and Neron-Severi group $NS(A)$ are isomorphic to $\ZZ$. Then, one possibility is to show that dual abelian variety has a natural polarisation of dual type, i.e. of type $(\frac{d_1d_n}{d_n},\frac{d_1d_n}{d_2}\ldots,\frac{d_1d_n}{d_1})$ \cite[Rem 5.2.9]{BL}. In particular, if $A$ is of type $(1,1,2)$, then $\hat{A}$ is of type $(1,2,2)$ hence (as $NS(A)$ is small) not isomorphic to $A$, see Theorem \ref{vgthm} for details. 

A second solution, that can be deduced from \cite[11.24]{vdg}, is to take a principally polarised abelian surface $A$ with $\End(A)=\ZZ$, a prime number $p$, a $p$-torsion point $P$ and $\left<P\right>$ the group generated by $P$, and consider the surface $B=A/\left<P\right>$ with the quotient map $\pi: A\to B$ that is of degree $p$. Then, the authors of \cite{vdg} show that $B$ does not admit a principal polarisation and we will adjust the construction to show non-existence of an isomorphism between $B$ and $\hat{B}$.
To see this, let us assume, by contrary, that there is an isomorphism $\phi:B\to\hat{B}$ and consider the following composition (where $\hat{\pi}$ is the dual homomorphism to $\pi$):
\begin{equation} \label{polpol} 
\xymatrix@R=1cm@C=1.5cm{
 A \ar[r]^{\pi}  & B  \ar[r]^{\phi} &\hat{B}  \ar[r]^{\hat{\pi}}&\hat{A}\simeq A
}
\end{equation}

Then $\hat{\pi}\circ\phi\circ\pi\in \End(A)$ and by computing degrees, we have that $\degi(\hat{\pi}\circ\phi\circ\pi) =p\cdot1\cdot p=p^2$ which gives a contradiction because $\End(A)=\ZZ$, and the degree of multiplication by $k\in \N$ is equal to $k^4$.

Both arguments show examples in some non-principal polarisation types, but not include all of them and they consider very general abelian varieties, i.e., one assumes that the Picard number of a variety is equal to $1$.

On the other extreme, if the Picard number is largest possible i.e. equal to $(\dim(A))^2$ then one shows that A is isomorphic to the product of (pairwise isogenous) elliptic curves with complex multiplication, hence it is always (abstractly) isomorphic to its dual abelian variety, see \cite[Thm 4.1]{HL} or \cite[Ex 5.6.10]{BL}.

The main aim of this note is to construct examples of abelian varieties in any dimension and any (non-principal) polarisation type that are not general and still not isomorphic to their duals. The result follows from Theorem \ref{mainthm} 
\begin{theorem*}[Thm 3.2]\label{mthm}
Let $(1,d_2,\ldots,d_n)$ be a sequence of positive integers satisfying $d_i|d_{i+1}$ for $i=2,\ldots,n-1$ and with $d_n>1$.
    Let $E$ be an elliptic curve such that there exists a primitive $d_n$-torsion $P\in E[d_n]$ such that $E/\left<P\right>$ is not isomorphic to $E$. Let $(B,\sL_B)$ be a $(d_2,\ldots,d_n)$ polarised abelian variety, such that $\Hom(E,B)=0$. Let $Q\in \ker\varphi_{\sL_B}\cap B[d_n]$ be a primitive $d_n$-torsion point.
    Then the quotient $A=(E\times B)/\left<(P,Q)\right>$ is a $(1,d_2,\ldots,d_n)$ polarised abelian variety such that $\hat{A}$ is not isomorphic to $A$.    
\end{theorem*}
The idea of the proof is to generalise a recent argument developed in \cite[\S 2]{BS}. As applications of the theorem we show explicit examples of abelian varieties not isomorphic to their duals constructed in Section \ref{SecExamples} and we show the following theorem.
\begin{theorem*}[Thm \ref{vgthm}]
    Let $A$ be a very general non-principally polarised abelian variety. Then $A$ is not isomorphic to its dual $\hat{A}$.    
\end{theorem*}
The idea of the proof is to show that dualisation is never an identity on the moduli space of non-principally polarised abelian varieties, hence a general abelian surface is not polarised isomorphic to its dual (i.e. there does not exist an isomorphism that preserves polarisations).  Then, assuming again that an abelian surface is very general, one gets that a pullback of a polarisation by a possible abstract isomorphism would need to be a polarisation, hence an isomorphism would be in fact polarised isomorphism that gives a contradiction. Note that this result has been proved for surfaces in \cite[Thm 1.5]{GH}, but up to our knowledge, it has not been described in the literature in full generality yet.

Note that for $n>3$ in \cite{HL} it has been shown that there is a gap in realizable Picard numbers, namely there are no abelian varieties with Picard number greater than $(\dim(A)-1)^2+1$ and less than $(\dim(A))^2$, hence the former number can be called a submaximal Picard number. In this context, we would like to point out that our construction includes examples of varieties with submaximal Picard number, see Example \ref{ex1}. 

Having the constructions described above we see that abelian varieties not isomorphic to their duals are quite common for non-principal type of polarisation. Moreover, they may have different Picard numbers, with the exception of the maximal one, where the variety is principally polarisable. Therefore, one could ask whether the existence of an isomorphism between an abelian variety $A$ and its dual $\hat{A}$ means that there exists a principal polarisation on $A$. In Section \ref{1313} we show that this is not true. Namely, in Example \ref{ex3}, we show that for 
 a very general $(1,3)$ polarised abelian surface $S$, the product
     $A=S\times\hat{S}$ does not admit a principal polarisation, yet it is isomorphic to its dual by switching the factors.

The paper is organised as follows. In Section \ref{2} we give a necessary background about abelian varieties. In Section \ref{SecMain} we provide a necessary result for construction of explicit examples of dual abelian varieties in Section \ref{SecExamples}. Finally in Section \ref{1313}, we construct a family of examples of dual abelian varieties which are not principally polarisable.

\subsection*{Acknowledgements}
The first author has been supported by the National Science Center, Poland, project ``Complex contact manifolds and geometry of secants'', 2017/26/E/ST1/00231. The second author has been supported by the Polish National Science Center project number 2018/30/E/ST1/00530. The authors would like to thank an anonymous referee for his helpful comments.

\section{Preliminaries}\label{2}
The aim of this section is to set up notation, give details and references for facts given in the introduction. As a main reference, we use \cite{BL}.

Every $n$-dimensional complex torus $A$ can be written as $A=\CC^n/\Lambda$ where $\Lambda$ is a lattice of rank $2n$. The existence of an ample line bundle on $A$ is equivalent to the existence of the Riemann Form, i.e a positive definite Hermitian form on $\CC^n$
with the imaginary part being integer valued on $\Lambda\times\Lambda$. In this case there exists a basis of $\Lambda$ that forms columns of so called big period matrix:
$$[Z\ D],$$
where $Z=[z_{ij}]$ satisfies $Z=Z^{t}$ and $\Ima Z>0$ and $D=\diag(d_1,\ldots,d_n)$. When needed, we denote the columns of the big period matrix by $f_1,\ldots,f_n,d_1e_1,\ldots,d_ne_n$.
Now, the Riemann form given by a matrix $(\Ima Z)^{-1}$ (in the canonical basis $e_1,\ldots,e_n$ of $\CC^n$) can be seen as the Neron-Severi class of a line bundle of type $(d_1,\ldots,d_n)$. The matrix $Z$ is called a small period matrix and the set of these is called the Siegel upperhalf space. For details, see \cite[Sec 8.1]{BL}. 

Now, \cite[Ex. 2.6.4]{BL} shows that $NS(A)=\ZZ$ unless the coefficients of $Z$ satisfy some (quadratic) equations with integral coefficients. 
Similarly, any endomorphism $f\in\Endo(A)$ can be lifted to a linear map $\sF:\CC^n\to\CC^n$ satisfying $\sF(\Lambda)\subset\Lambda$, hence for $\QQ(z_{11},\ldots,z_{nn})$ of maximal transcendental degree over $\QQ$ one has that $\Endo(A)=\ZZ$.
A very general abelian variety is an element of the moduli space outside countably many subloci of smaller dimensions. The above argument shows the existence of an abelian variety $A$ with $NS(A)=\ZZ$ and $\End(A)=\ZZ$ and in particular we assume that a very general abelian variety satisfies these conditions.

An element $x\in A$ is called a $k$-torsion point if $kx=x+\ldots+x=0$. It can be seen as an element of the kernel of multiplication by $k$ on $A$ and in particular it lies in the image of $\frac 1k\Lambda$. Hence the set of such elements, denoted by $A[k]$, is isomorphic to $\ZZ_k^{2n}$. By $\left<x\right>$ we denote the (finite) subgroup generated by $x$.

An isogeny between abelian varieties is a surjective homomorphism with a finite kernel. There are two useful constructions of these. First, for a finite group $G$ embedded in an abelian variety $A$, we have a natural quotient isogeny $A\to A/G$. Secondly, there is so called polarising isogeny
$\varphi_{\sL}:A\to \hat{A}$ given by $\varphi_{\sL}(x)=t^*_x(\sL)\otimes \sL^{-1}$, where $t_x$ is translation by $x$. The proof that indeed $\varphi_{\sL}$ is an isogeny that does not depend on a line bundle within the polarisation can
be found in \cite[Prop 2.5.3]{BL}. It is also worth noting that if we consider $\QQ$-homomorphisms, i.e. elements of $\Hom(X,Y)\otimes\QQ$, then isogenies become invertible, see \cite[Prop 1.2.6]{BL}.

If $(A,c_1(\caL_A)), (B,c_1(\caL_B))$ are polarised abelian varieties then a polarised isogeny $f:A\to B$ is an isogeny that preserves polarisations, i.e. $c_1(f^*(\caL_B))=c_1(\caL_A)$. Having an isogeny and a polarisation on $B$, an obvious way to make it a polarised isogeny is to choose $c_1(f^*(\caL_B))$ as polarisation on $A$. Note that a pullback polarisation may not satisfy $d_1=1$.

As we have seen, a very general abelian variety is in particular simple, i.e. it does not contain any abelian subvariety apart from $\{0\}$ and itself. Now, assume that there exists an abelian subvariety $\iota:B\hookrightarrow A$.
Firstly, we can define so called restricted polarisation (that is simply choosing a polarising line bundle $\sL$ and taking a class of $\sL|_B$) with the restricted polarisation type. Note that the restricted polarisation type may satisfy $d_1>1$ and we are not allowed to change it to get $d_1=1$.

Having a homomorphism $\iota:B\to A$ one defines the dual morphism $\hat{\iota}:\hat{A}\to\hat{B}$ that after the identification $\hat{B}=Pic^0(B)$ coincides with $\iota^*$.
Now, one defines the associated symmetric idempotent to be
$$\epsilon_B=\iota\varphi_{\sL|_B}^{-1}\hat\iota\varphi_{\sL}\in\Endo(A)\otimes\QQ$$
and the Norm endomorphism $$Nm_B=e(B)\epsilon_B\in\Endo(A)$$
where $e(B)$ is, so called, exponent of $B$, i.e. the exponent of the finite group $\ker\varphi_{\sL|_B}$, see \cite[Chapter 5]{BL}. 

Since the assignment between subvarieties and associated symmetric idempotents is bijective, one can define the complementary abelian subvariety $C\subset A$ by using the idempotent $\epsilon_C=1-\epsilon_B$. In such a case we write $A=B\boxplus C$, see \cite[\S5]{BO17}. Moreover, we have the addition map $\mu:B\times C\to A$ that is a polarised isogeny between $A$ and $B\times C$ with the product polarisation, see Poincare's Reducibility Theorem \cite[Thm. 5.3.5 and Cor. 5.3.6]{BL}. 

Now, if $X=\CC^n/\Lambda_X, Y=\CC^m/\Lambda_Y$ then every homomorphism $f:X\to Y$ can be lifted to a linear map $\sF:\CC^n\to\CC^m$ satisfying $\sF(\Lambda_X)\subset\Lambda_Y$ that is called the analytic representation of $f$ and in particular, by choosing a suitable bases it can be written as an $m\times n$ matrix. That is why, the matrix of a Riemann form can also be seen as a matrix of a polarising isogeny $\phi_{\sL}$, see \cite[Lem. 2.4.5]{BL}.

Lastly, having a group $G$ and an isogeny $f:A\to A/G$ we want to be able to decide what are (the types of) polarisations on $A,A/G$ such that $f$ becomes polarised. The condition can be retrieved from \cite[Cor. 6.3.5]{BL} and in our case, when $G$ is cyclic, we only need that $G\subset \ker\varphi_{\sL}$ since cyclic groups are isotropic with respect to any symplectic form.

\section{Main result}\label{SecMain}
We start this section by recalling a useful lemma that is an easy generalisation of \cite[Prop 2.3]{BS} or follows from \cite[Prop 5.2]{BO17}.
\begin{lemma}\label{onlyE}
    Let $E$ be an elliptic curve. If $A=E\boxplus B$ with $\Hom(E,B)=0$ then the only elliptic curves isogenous to $E$ that are embedded in $A$ is $E$ itself.
\end{lemma}
\begin{proof}
    Let $F\subset A$ be an elliptic curve isogenous to $E$. Certainly $Nm_B|_F\in \Hom(F,B)$, so by assumptions it is equal to $0$.
    By \cite[p.125]{BL}, we have the following equality $e(E)e(B)id=e(B)Nm_E+e(E)Nm_B$, so by taking the image of both sides restricted to $F$, we get
    $F=im(Nm_E|_F)\subset E$, so $F=E$.
\end{proof}

Now, we are ready to prove the main theorem that will be applied to explicit examples in the next section.
\begin{thm}\label{mainthm}
Let $(1,d_2,\ldots,d_n)$ be a sequence of positive integers satisfying $d_i|d_{i+1}$ for $i=2,\ldots,n-1$ and with $d_n>1$.
    Let $E$ be an elliptic curve such that there exists a primitive $d_n$-torsion $P\in E[d_n]$ such that $E/\left<P\right>$ is not isomorphic to $E$. Let $(B,\sL_B)$ be a $(d_2,\ldots,d_n)$ polarised abelian variety, such that $\Hom(E,B)=0$. Let $Q\in \ker\varphi_{\sL_B}\cap B[d_n]$ be a primitive $d_n$-torsion point.
    Then the quotient $A=(E\times B)/\left<(P,Q)\right>=E\boxplus B$ is a $(1,d_2,\ldots,d_n)$ polarised abelian variety such that $\hat{A}$ is not isomorphic to $A$.    
\end{thm}
\begin{proof}
    Let $\sL_E$ be a principal polarisation on $E$ and $\sL_B$ be a polarisation of type $(d_2,\ldots,d_n)$ on $B$. Consider $E\times B$ with the product polarisation $\sL$ being the class of the bundle $\sL_E^{\otimes d_n}\boxtimes \sL_B$ (where $\boxtimes$ means the tensor product of pullbacks of line bundles). Since the polarising isogeny is the product of polarising isogenies, its kernel is the product of kernels, so its type is $(d_2,\ldots,d_n,d_n)$. Let $P\in E,\ Q\in B$ be as in the statement of the theorem. Then $\left<(P,Q)\right>$ is an (isotropic) subgroup of $\ker\varphi_{\sL} $ of order $d_n$ on $E\times B$. 
    
    Denote by $A=E\times B/\left<(P,Q)\right>$. 
    By \cite[Prop 6.3.5]{BL}, we get that the quotient map $\pi: E\times B\to A$ is a polarised isogeny and in particular there exists a polarising line bundle $\sL_A$ of type $(1,d_2,\ldots,d_n)$ on $A$.  The type is as stated because $\ker\varphi_{\sL_A}=\pi((\ker\pi)^{\perp})$, where $\perp$ stands for the symplectic complement in $\ker\varphi_{\sL}$ (see \cite[Prop 6.3.3]{BL}).
    
    To show that $A=E\boxplus B$ one can compute associated idempotents or use the following trick. Since $\Hom(E,B)=0$, one can prove that $E$ is the only elliptic curve in $E\times B$ that is not contained in $B$, see the proof of Lemma \ref{onlyE}. Therefore, the same applies to $A$, because one can take preimages by $\pi$. In particular, $E$ is the only curve that satisfies $E+B=A$, so $E$ has to be the complementary abelian subvariety to $B$ in $A$.

    Now, we compute that $\ker\varphi_{\sL_A}\cap E=\left<P\right>$, so $\varphi_{\sL_A}(E)=E/\left<P\right>\subset \hat{A}$. 
    As stated before, by Lemma \ref{onlyE} we get that $E=\pi(E\times\{0\})$ is the only elliptic curve isogenous to $E$ that is embedded in $A$.    
    This finishes the proof because if there were an isomorphism between $\hat{A}$ and $A$, we would have got $E/\left<P\right>\subset A$ and since by assumption it is not isomorphic to $E$, we get a contradiction with Lemma \ref{onlyE}. 
\end{proof}

\begin{remark}\label{remtauoverp}
We would like to explore the condition from the main theorem that the elliptic curve $E/\left<P\right>$ is not isomorphic to $E$. 
Let $E$ be given by periods $1,\tau$ with $Im(\tau)>0$. By a possible change of periods, we can assume that $P$ is the image of a point $\frac\tau n\in\CC$, where $n>1$ is the order of $P$. Then, the curve $E/\left<P\right>$ is given by periods $1,\frac\tau n$. The existence of an isomorphism between $E$ and $E/\left<P\right>$ is equivalent to the existence of a matrix $\left[\ ^a _b\ ^c _d\right]\in Sl(2,\ZZ)$ such that 
\begin{equation}\label{sl}
    \frac{a\tau+b}{c\tau+d}=\frac{\tau}{n}.
\end{equation} 
Now, if $c=0$ then $d=\pm 1$ and the solution to the above equation is real (in fact even rational), so the solution cannot be a period. Moreover, $c\neq 0$ gives at most two solutions for possible periods $\tau$. This shows that there are (at most) countably many solutions to Equation \ref{sl} and in particular for a very general $\tau$ we have $E/\left<P\right>$ is not isomorphic to $E$.

An explicit example when $E/\left<P\right>$ is actually isomorphic to $E$ can be given by the period $\tau_E=\sqrt{-2}$ and $P=\frac{\tau_E}{2}$. Then one gets that  $\tau_{E/\left<P\right>}=\frac{\sqrt{-2}}{2}$ is an isomorphic elliptic curve because $\tau_{E/\left<P\right>}=\frac{-1}{\tau_E}$.  
    
\end{remark}

\section{Explicit examples}\label{SecExamples}
This section shows some applications of Theorem \ref{mainthm} and provides explicit examples.

For our purposes and to make things fully explicit we can choose $\tau_E=\pi i$ and $d_n>1$. Now, since we want the polarisation to be of type $d_n$ we choose the lattice to be generated by $d_n\pi i\ZZ+d_n\ZZ$. Now, choose $P=1\in E[d_n]$. Then $\tau_{E/\left<P\right>}=d_n\pi i$ because the second generator is $1$ and by Remark \ref{remtauoverp} they are not isomorphic to each other.

Now, if an elliptic curve $F$ has complex multiplication (for example $\tau_F= i$) then certainly $E,F$ are not isogenous. Now, consider $B=F^{n-1}$ with the product polarisation of type $(d_2,\ldots,d_n)$ with a $d_n$--torsion point $Q$ 
being the image of the vector $-e_n$ (note that, by construction of the period lattice for polarisation of type $(d_2,\ldots,d_n)$ one gets that $d_ne_n$ belongs to the lattice, so  the image of $-e_n$ is indeed a $d_n$-torsion point and it lies in the kernel of the polarising isogeny).

Then $E\times F^{n-1}$ can be given by a period matrix
$$\Lambda_{E\times F^{n-1}}=\begin{bmatrix}
d_n\tau_E & \cdots & \cdots&0&d_n&\ldots&\ &0\\
0 & \tau_F & 0&\vdots& 0&d_2& \ & \vdots\\
\vdots & \ & \ddots&0&0&\ &\ddots&\ \\
0 & \ldots & \ &\tau_F&0&0&\ldots &d_n\\
\end{bmatrix}. $$

\begin{example}\label{ex1} Let $(P,Q)$ be the image of $(1,0,\ldots,0,-1)$ as above. Then, an abelian variety $A$ obtained by dividing $E\times F^{n-1}$ by $\left<(P,Q)\right>$ is a $(1,d_2,...,d_n)$-polarised $n$-fold not isomorphic to its dual.
\end{example}
\begin{proof}
To get a symmetric period matrix of $A$, one needs to find a symplectic change of basis. We do this by changing the basis of $\CC^n$ to $e_1-e_n,e_2,\ldots,e_n$ and we change the basis of a lattice to $f_1,\ldots,f_{n-1},f_n+f_1, d_ne_1-d_ne_n,d_2e_2,\ldots,  d_ne_n$.

Then $(P,Q)$ is the image of $e_1-e_n$, hence one gets the following period lattice of $A$:
$$\Lambda_A=\begin{bmatrix}
d_n\tau_E & 0 & 0&d_n\tau_E&1&\ldots&\ &0\\
0 & \tau_F & 0&\vdots& 0&d_2& \ & \vdots\\
0 & \ & \ddots&0&0&\ &\ddots&\ \\
d_n\tau_E & 0 & 0 &d_n\tau_E+\tau_F&0&0&\ldots &d_n\\
\end{bmatrix}, $$
with an elliptic curve $E$ embedded via the analytic representation of the form  $z\mapsto (z,0\ldots,0,z)$. Certainly $\Hom(E,F^{n-1})=0$ because $E,F$ are non-isogenous, so by Theorem \ref{mainthm}, $A=\CC^n/\Lambda_A$ satisfies the claim.
\end{proof}
\begin{remark}
    Note that the Picard number of $A$ from Example \ref{ex1} attains the submaximal value $(n-1)^2+1$, see \cite[Thm. 4.2]{HL}.  
\end{remark}
\begin{example}\label{ex2}
    Let $E$ be an elliptic curve with a period $d_n,\ d_n\tau_E$ and $P\in E[d_n]$ be the image of $1\in\CC$. Let $B$ be a very general abelian $n-1$ fold of type $(d_2,\ldots,d_n)$ and $Q$ any $d_n$ torsion point in the kernel of the polarising isogeny. Then $A=E\times B/\left<(P,Q)\right>$ is a $(1,d_2,\ldots,d_n)$ polarised abelian variety such that $\hat{A}$ is not isomorphic to $A$.  
\end{example}
\begin{proof}
 $B$ is simple so $\Hom(E,B)=0$, hence the assumptions of Theorem \ref{mainthm} are satisfied. In a similar manner as in Example \ref{ex1}, denoting a small period matrix of $B$ by $[b_{i,j}]_{i,j=1,\ldots,n-1}$ and chosen so that $Q$ is the image of $e_n$ one finds a period matrix of $A$ to be: 
$$\Lambda_A=\begin{bmatrix}
d_n\tau_E & 0 & 0&d_n\tau_E&1&\ldots&\ &0\\
0 & b_{1,1} & \cdots&b_{1,n-1}& 0&d_2& \ & \vdots\\
0 & \vdots & \ddots&\vdots&0&0&\ddots&\ \\
d_n\tau_E & b_{n-1,1} & \ldots &d_n\tau_E+b_{n-1,n-1}&0&0&\ldots &d_n\\
\end{bmatrix}. $$
\end{proof}
\begin{remark}
    Note that the Picard number of $A$ from Example \ref{ex2} equals 2. 
\end{remark}

\begin{remark}
We have shown two families of examples that follows from Theorem \ref{mainthm}. However, it may be worth noting that for any abelian variety $B$ one can construct an abelian variety $A$ not isomorphic to its dual that contains $B$ as an abelian subvariety of codimension 1. To see this, one uses the fact that there are at most countably many elliptic curves embedded in $B$ (see \cite[Ex. 1.5.1]{BL}), hence at most countably many elliptic curves that satisfy $\Hom(E,B)\neq 0$. In particular there exists an elliptic curve $E$ such that $\Hom(E,B)=0$ and if $B$ is principally polarised, one uses $d_n$ times the principal polarisation to get $d_n>1$. Then the assumptions of Theorem \ref{mainthm} are satisfied and one can adjust the construction of Example \ref{ex2} to get a period matrix of the variety $A$.
\end{remark}

We can also apply the main result to the case of very general abelian varieties. We will show the following theorem:
\begin{theorem}\label{vgthm}
    Let $A$ be a very general non-principally polarised abelian variety. Then $A$ is not isomorphic to its dual $\hat{A}$.
\end{theorem}
\begin{proof}
    Let $\Psi:\cA_g^D\to\cA_g^{\hat{D}}$ be the morphism that takes an abelian variety to its dual and let $A\in\cA_g^D$ be of Picard rank 1.
    If $D$ is not equal to $\hat{D}$ then certainly $\hat{A}$ is not polarised isomorphic to $A$.
    If $D=\hat{D}$ then, by Theorem \ref{mainthm}, we see that $\Psi$ is not the identity morphism, hence, by the very definition of the moduli space, there exists an open subset of $\cA_g^D$ where $\hat{A}$ is not polarised isomorphic to $A$.

    Now, we are ready to use the same trick as in the introduction. Let $(A,\sL_A)$ be a polarised abelian variety with Picard rank 1. Since $\hat{A}$ is isogenous to $A$, we know that its Picard rank is 1 too, given by $H_{\hat{A}}$. If there was an abstract isomorphism $f:\hat{A}\to A$ then by choosing a polarisation $f^*(H_A)$ on $\hat{A}$ one would make $f$ a polarised isomorphism between   $(A,H_A)$ and $(\hat{A},f^*{H_A})$. Because $A$ has Picard rank 1, we know that $f^*H_{A}$ is a multiple of $H_{\hat{A}}$ and since $f$ was an isomorphism, their degrees coincide, so they are equal. This shows that in particular, $f$ was a polarised isomorphism, which is a contradiction with our set-up. 
\end{proof}

\section{Dual non principally polarisable abelian varieties}\label{1313}
The fact that \textit{if an abelian variety is principally polarised then it is isomorphic to its dual} is widely used. 
In this section we will show that the converse implication is not true. In particular we show an explicit example of a non-principally polarisable abelian variety which is isomorphic to its dual. 

Before doing so, we need to introduce some notation and preliminary results.
We start with the following characterisation of polarising isogenies:
\begin{lemma}\cite[Thm 2.5.5]{BL}
   Suppose $A=\CC^n/\Lambda_A$ and $f:A\to\hat{A}$ with the analytic representation $\mathcal{F}$. 
   Then $f=\varphi_{\sL}$ for some line bundle $\sL\in Pic(A)$ if and only if $\sF$ is Hermitian.
  
\end{lemma}
\begin{corollary}
In the previous theorem, we have that $f$ is polarising isogeny if and only if $\sF$ is positive definite Hermitian form.
\end{corollary}

Let 
$S$ be a $(1,3)$ polarised abelian surface given by a period matrix
$$L_S=\begin{bmatrix}
a & b & 1&0\\
b & c & 0&3\\
\end{bmatrix},$$
where $a,b,c\in\CC$ are algebraically independent over $\QQ$ and such that the imaginary part of the small period matrix is positive definite. This assures that $S$ is a very general abelian surface.

Note that by \cite[\S 2.4 p.34]{BL} the dual lattice to $\Lambda$ is defined as $\{l\in Hom_{\bar{\CC}}(\CC^2,\CC): \forall x\in\Lambda,\  Im(l(x))\in\ZZ\}$. Using the dual basis to the canonical one, the dual lattice is generated by the vectors: 
$$\begin{bmatrix}
a & \frac{b}{3} & 1&0\\
b & \frac{c}3 & 0&1\\
\end{bmatrix},$$
To get a symmetric small period matrix, we divide the first canonical vector basis by $3$ to get the big period lattice to be of the form 
$$L_{\hat{S}}=\begin{bmatrix}
3a & b & 3&0\\
b & \frac{c}3 & 0&1\\
\end{bmatrix}.$$
In principle, we should switch coordinates to get a type $(1,3)$ but since it is just a matter of convention we will leave the coordinates as they are. Note that $\widehat{\widehat{S}}=S$ and our procedure yields $L_{\hat{\hat{S}}}=L_S$.
\begin{example}\label{ex3}
Let $A=S\times \hat{S}$ with a big period lattice given by the matrix:
$$
L_A=\begin{bmatrix}
a & b & 0&0&1&0&0&0\\
b & c & 0&0&0&3&0&0\\
0 & 0 & 3a&b&0&0&3&0\\
0 & 0 & b&\frac{c}{3}&0&0&0&1\\
\end{bmatrix}. $$
Then $A$ is isomorphic to $\hat{A}$ but does not admit a principal polarisation.    
\end{example}
\begin{proof}
    Certainly $\hat{A}$ is isomorphic to $A$ because $$\hat{A}=\widehat{S\times\hat{S}}\simeq\hat{S}\times\hat{\hat{S}}\simeq\hat{S}\times S\simeq S\times\hat{S}=A$$

The fact that there doesn't exist a principal polarisation on $A$ follows from Lemma \ref{54}.
\end{proof}
\begin{lemma}\label{54}
There does not exist a $4\times 4$ matrix $H$ of a Hermitian positive definite form such that the lattice generated by the columns of $HL_A$ coincide with the lattice generated by the columns of the matrix $$L_{\hat{A}}=\begin{bmatrix}
3a & b & 0&0&3&0&0&0\\
b & \frac{c}{3} & 0&0&0&1&0&0\\
0 & 0 & a&b&0&0&1&0\\
0 & 0 & b&c&0&0&0&3\\
\end{bmatrix}.$$
\end{lemma}
\begin{proof}
Let us assume that such $H$ exists.
By assumptions, $H$ is Hermitian positive definite matrix and let $h_{ij}$ be its entries. Firstly, we want to show that $H$ has integer entries, hence it is a positive defined symmetric matrix over $\mathbb{Z}$. To see that, consider the first row of the matrix $HL_A$. All its entries, in particular $h_{11}, 3h_{12},3h_{13}, h_{14}$ 
need to be integer combination of $3a,b,3$ which are algebraically independent. Combining this with the fact that the remaining entries of the first row, namely $h_{11}a+h_{12}b, h_{11}b+h_{12}c, 3h_{13}a+h_{14}b, h_{13}b+h_{14}\frac{c}{3}$ are also integer combination of $3a,b,3$ we obtain the claim. 

Now, again analyzing the first row and taking into account that $c$ is algebraically independent from $3a,b,3$, we immediately see that $h_{12}=h_{14}=0$. Moreover, $h_{11}$ must be a multiple of $3$ and since $H$ is positive definite, we get that $h_{11}=3k, k\in\ZZ_+$. Repeating analogous argument with the remaining rows and including the symmetry of $H$ (that comes from the fact $H$ is Hermitian with real coordinates) we get that 
$$H=\begin{bmatrix}
3k & 0 & h_{31}&0\\
0& h_{22} & 0&h_{24}\\
h_{13} & 0 & h_{33}&0\\
0 & h_{24} & 0&3m\\
\end{bmatrix}.$$

Now we will use the fact that the columns of $HL_A$ generate over $\mathbb{Z}$ columns of $L_{\hat{A}}$. In particular there exist two non-zero integer pairs $(c_1, c_2)$ and $(d_1,d_2)$ such that 

$$c_1\begin{bmatrix}
3ka\\
h_{22}b\\
h_{13}a\\
h_{24}b \\
\end{bmatrix} +c_2\begin{bmatrix}
3h_{13}a\\
h_{24}b\\
3h_{33}a\\
3mb \\
\end{bmatrix} =\begin{bmatrix}
3a\\
b\\
0\\
0\\
\end{bmatrix}\ \text{ and }\ d_1\begin{bmatrix}
3ka\\
h_{22}b\\
h_{13}a\\
h_{24}b \\
\end{bmatrix} +d_2\begin{bmatrix}
3h_{13}a\\
h_{24}b\\
3h_{33}a\\
3mb \\
\end{bmatrix} =\begin{bmatrix}
0\\
0\\
a\\
b\\
\end{bmatrix}.$$

Fixing the matrix coefficients and treating $(c_1, c_2)$ and $(d_1,d_2)$ as integer variables, from each of these equations we get a system of two homogeneous linear equations. By assumption, these systems have non-zero solutions, therefore, from the first of them we get $(h_{13},3h_{33})=\alpha (h_{24},3m)$, for $\alpha\in\QQ\setminus\{0\}$. Substituting this to $d_1h_{13}+3d_2h_{33}=1$ and using that $d_1h_{24}+3d_2m=1$ we get that $\alpha=1$. Similarly from the second system of equations we get
$ (k,h_{13})= (h_{22}, h_{24})$ and therefore $$H=\begin{bmatrix}
3k & 0 & h&0\\
0& k & 0&h\\
h & 0 & m&0\\
0 & h & 0&3m\\
\end{bmatrix}.$$

To finish the proof we need to consider the condition that the columns of $HL_A$ generate over $\mathbb{Z}$ the last four columns of $L_{\hat{A}}$. A non-zero rational coefficient in the first and the third row appears only in the 5th and the 7th column of $HL_A$ (the coefficients in this case are in fact integers). Hence there exist integer pairs $(c_1',c_2')$ and $(d_1',d_2')$ such that

$$c_1'\begin{bmatrix}
3k\\
0\\
h\\
0\\
\end{bmatrix} +c_2'\begin{bmatrix}
3h\\
0\\
3m\\
0 \\
\end{bmatrix} =\begin{bmatrix}
3\\
0\\
0\\
0\\
\end{bmatrix}\ \text{ and }\ d_1'\begin{bmatrix}
3k\\
0\\
h\\
0\\
\end{bmatrix} +d_2'\begin{bmatrix}
3h\\
0\\
3m\\
0 \\
\end{bmatrix} =\begin{bmatrix}
0\\
0\\
1\\
0\\
\end{bmatrix}.$$

From equations $d_1'h+3d_2'm=1$ and $c_1'k+c_2'h=1$ we see that $d_1',d_2'$ are coprime as well as $k,h$. Hence the equation $d_1'k+d_2'h=0$ rephrased as $d_1'k=-d_2'h$ implies that $(d_1',d_2')=\pm(h,-k)$.

Analyzing principal minors of the positive definite matrix $H$ we get that $k>0$ and $3k^2m-h^2k>0$ hence $3km-h^2>0$ and by $d_1'h+3d_2'm=1>0$ we get that $(d_1',d_2')=(-h,k)$. Finally this implies that $3km-h^2=1$, but $h^2\equiv 0 \text{ or } h^2\equiv 1(\mod 3)$ which gives a contradiction.
\end{proof}
\begin{remark}
    The above lemma is also correct for $(1,d)$, where $d$ is a number such that $-1$ is not a square in $\mathbb{Z}_d$. It is worth to note that a similar argument can be used to show that $A^4\times\hat{A}^4$ is always principally polarisable (for any abelian variety $A$). It is commonly known as a Zarhin's trick, see \cite[Thm 11.29]{vdg} or \cite{Z}.   
\end{remark}

\end{document}